\documentclass[a4paper]{article}

\usepackage{amsfonts}
\usepackage{amsmath}
\usepackage{mathrsfs}
\usepackage{amssymb,color}

\usepackage[latin1]{inputenc}
\usepackage{amsthm}
\usepackage{graphicx}
\usepackage{enumerate}
\newtheorem{theorem}{Theorem}[section]
\newtheorem{proposition}[theorem]{Proposition}
\newtheorem{definition}[theorem]{Definition}
\newtheorem{lemma}[theorem]{Lemma}
\newtheorem{remark}[theorem]{Remark}
\newtheorem{corollary}[theorem]{Corollary}

\begin{document}
	
	\title{\textbf{Constant sign solution for simply supported beam equation with non-homogeneous boundary conditions.\footnote{Partially supported by Ministerio de Economía y Competitividad, Spain and FEDER, project MTM2013-43014-P.} }}
	\date{}
	\author{Alberto Cabada   \;and  Lorena Saavedra\footnote{Supported by FPU scholarship, Ministerio de Educaci\'on, Cultura y Deporte, Spain.}\\Departamento de Análise Matemática,\\ Facultade de  Matemáticas,\\Universidade Santiago de Compostela,\\ Santiago de Compostela, Galicia,
		Spain\\
		alberto.cabada@usc.es, lorena.saavedra@usc.es
	} 
	\maketitle
	\begin{abstract}
	The aim of this paper is to study the following fourth-order operator:
	\begin{equation*}
	T[p,c]\,u(t)\equiv u^{(4)}(t)-p\,u''(t)+c(t)\,u(t)\,,\quad t\in I\equiv [a,b]\,,
	\end{equation*}	
	coupled with the non-homogeneous simply supported beam boundary conditions:
		\begin{equation*}
		u(a)=u(b)=0\,,\quad u''(a)=d_1\leq0\,,\ u''(b)=d_2\leq 0\,.	
		\end{equation*}
		
		First, we prove a result which makes an equivalence between the strongly inverse positive (negative) character of this operator with the previously introduced boundary conditions and with the homogeneous boundary conditions, given by:
		\begin{eqnarray}
		\nonumber T[p,c]\,u(t)=h(t)(\geq0)\,,\\\nonumber
		u(a)=u(b)=u''(a)=u''(b)=0\,,
		\end{eqnarray}

		 Once that we have done that, we prove several results where the strongly inverse positive (negative) character of $T[p,c]$ it is ensured.
		 
		 Finally, there are shown  a couple of result which say that under the hypothesis that $h>0$, we can affirm that the problem for the homogeneous boundary conditions has a unique constant sign solution.		
		\end{abstract}
		
		\section{Introduction}
		The study of different fourth order differential equations coupled with the boundary conditions:
			\begin{equation}\label{Ec::cf}
			u(a)=u(b)=u''(a)=u''(b)=0\,,
			\end{equation}
			has been wider treated along the literature.
			
			For instance, in  \cite{Dra} and \cite{Dra1}, there are obtained sufficient conditions which ensure that the problem
			\begin{equation}\label{Ec::p0}
			u^{(4)}(t)+c(t)\,u(t)=h(t)(\geq 0)\,,
			\end{equation} 
			coupled with the homogeneous boundary conditions \eqref{Ec::cf} has a unique constant sign solution on the interval $[0,1]$. Both papers improve previous results obtained in  \cite{cacisa} and \cite{Sch} for the homogeneous case. We note that in \cite{cacisa},  different non-homogeneous boundary conditions, on the line of \eqref{Ec::cfnh}, are considered.
			
			
			In \cite{CaSaa2}, the strongly inverse positive (negative) character of the operator $u^{(4)}(t)+p_1(t)\,u^{(3)}(t)+p_2(t)\,u''(t)+M\, u(t)$ coupled with  boundary conditions \eqref{Ec::cf}, where $p_1\in C^{3}(I)$ and $p_2\in C^2(I)$, is determined by the spectrum of suitable related boundary conditions.
			
			The study of this kind of problems is very important, since they are used to model  different kind of bridges. In \cite{Dra2}, there are shown several examples of bridges and its mathematical models. Even though most of them there are non- linear problems, in order to study them, it is very important to know first the linear part of them. In particular the fact that the displacement of the bridge occurs in the same direction as the external force is fundamental in order to ensure the stability of the considered structure.
			
			In \cite{cacisa, Liu}  the existence of one or multiple positive solutions of some suitable non-linear problems are considered. The used tools are strongly involved with the constant sign of the related Green's function.
 
 At first, in Section 3, we  obtain a result which proves that  $T[p,c]$ is strongly inverse positive (negative) in $X$ (defined beloww and correspondent to the homogeneous boundary conditions \eqref{Ec::cf}) if, and only if, it is also strongly inverse positive (negative) on the following space (which corresponds to the non-homogeneous boundary conditions):
 \begin{equation}\label{Ec::Xt}\tilde X =\{u\in C^4 (I)\ \mid \ u(a)=u(b)=0\,,\ u''(a)\leq 0\,,\ u''(b)\leq 0\}\,.\end{equation}

Then,  we study the following fourth-order operator:
\begin{equation}\label{Ec::T4}
T[p,c]\,u(t)\equiv u^{(4)}(t)-p\,u''(t)+c(t)\,u(t)\,,\quad t\in I\,,
\end{equation}
	on the following space of definition:
	\begin{equation}\label{Ec::esp}
	X=\left\lbrace u\in C^4(I)\quad\mid\quad u(a)=u(b)=u''(a)=u''(b)=0\right\rbrace \,,
	\end{equation}
	which corresponds to the homogeneous boundary conditions \eqref{Ec::cf}.
	
	Here $p\in\mathbb{R}$ and $p\geq 0$.
	
	Realize that problem \eqref{Ec::p0}, studied in \cite{Dra, Dra1}, is a particular case of \eqref{Ec::T4} with $p=0$.

	In Section 4, we formulate the variational approach of problem 
	\begin{equation}
	\label{Ec::Op}
	T[p,c]\,u(t)=h(t), \quad t\in I,
	\end{equation} coupled with the boundary conditions given in \eqref{Ec::cf} and we obtain different previous results which will be used along the paper.
	
	In section 5, it will be obtained sufficient conditions to ensure that the problem \eqref{Ec::Op}, \eqref{Ec::cf} has a unique solution. Moreover, we will verify that this property also warrants a unique solution of problem \eqref{Ec::Op}, with the non-homogeneous boundary conditions:	
	\begin{equation}
	\label{Ec::cfnh}
	u(a)=u(b)=0\,,\quad u''(a)=d_1\leq 0\,,\quad u''(b)=d_2\leq 0\,,
	\end{equation}
	with $d_1$ and $d_2$ arbitrary non positive constants.
	
	In fact, section 5 is devoted to obtain sufficient conditions that warrant that the operator $T[p,c]$ is either strongly inverse positive in $\tilde X$ or strongly inverse negative in $\tilde X$.
	
	Finally, in Section 6, we obtain different conditions for functions $h>0$ and $c$ that ensure that the unique solution of the problem \eqref{Ec::Op}, \eqref{Ec::cf} is either positive or negative.

\section{Preliminaries}
In this section, we introduce several tools and results which are going to be used along the paper.
 
 We consider a  general $n^{\rm th}-$ order linear operator 
  \begin{equation}\label{Op::n}
 L_n[M]\,u(t)\equiv u^{(n)}(t)+p_1(t)\,u^{(n-1)}(t)+\cdots+ p_{n-1}(t)\,u'(t)+(p_n(t)+M)\,u(t)\,,
 \end{equation} 
 with $t \in I$ and $p_k\in C^{n-k}(I)$, $k=1, \ldots,n$.
 
\begin{definition}
	The $n^{\rm th}-$ order linear differential equation 
	\begin{equation}\label{Ec::n}
	L_n[M]\,u(t)=0\,,\ t\in I,
	\end{equation}
	is said to be disconjugate on $I$ if every non trivial solution has less than $n$ zeros at $I$, multiple zeros being counted according to their multiplicity.
\end{definition}

	We introduce a definition to our particular problem \eqref{Ec::T4} in the space $X$.
	\begin{definition}
		The operator $T[p,c]$ is said to be strongly inverse positive (strongly inverse negative) in $X$, if every function $u\in X$ such that $T[p,c]\,u\gneqq0$ in $I$, satisfies $u>0$ ($u<0$) on $(a,b)$ and, moreover $u'(a)>0$ and $u'(b)<0$ ($u'(a)<0$ and $u'(b)>0$).
	\end{definition}
	
	Let us denote $g_{p,c}$ the related Green's function to operator $T[p,c]$ in $X$. Next result, proved in \cite{CaSaa2}, shows a relationship between the Green's function's sign and the previous definition.
	
	\begin{theorem}\label{T::14}
		Green's function related to  operator $T[M]$ in $X$ is positive (negative) a.e on $(a,b)\times (a,b)$ and, moreover, $\frac{\partial }{\partial t}g_{p,c}(t,s)_{\mid t=a}>0$ and $\frac{\partial }{\partial t}g_{p,c}(t,s)_{\mid t=b}<0$ ($\frac{\partial }{\partial t}g_{p,c}(t,s)_{\mid t=a}<0$ and $\frac{\partial }{\partial t}g_{p,c}(t,s)_{\mid t=b}>0$) a.e. on $(a,b)$, if, and only if,   operator $T[M]$  is strongly inverse positive (strongly inverse negative) in $X$.
	\end{theorem}

\begin{itemize}
	\item Let $\lambda_1^p>0$ be the least positive eigenvalue of $T[p,0]$ in $X$.
		\item Let  $\lambda_2^p<0$ be the maximum between:
			\begin{itemize}
			
				\item[] ${\lambda_2^p}'<0$, the biggest negative eigenvalue of $T[p,0]$ in 
				\begin{equation*} X_1=\left\lbrace u\in C^4 (I)\quad\mid\quad u(a)=u(b)=u'(b)=u''(b)=0\right\rbrace \,,\end{equation*} 
				\item[] ${\lambda_2^p}''<0$, the biggest negative eigenvalue of $T[p,0]$ in 
				\begin{equation*}X_3=\left\lbrace u\in C^4 (I)\quad\mid\quad u(a)=u'(a)=u''(a)=u(b)=0\right\rbrace \,.\end{equation*} 
			\end{itemize}

		\item Let $\lambda_3^p>0$ be the minimum between:
			\begin{itemize}
				\item[] ${\lambda_3^p}'>0$, the least positive eigenvalue of $T[p,0]$ in\begin{equation*} U=\left\lbrace u\in C^4 (I)\quad\mid\quad u(a)=u'(a)=u(b)=u''(b)=0\right\rbrace \,,\end{equation*}

				\item[] ${\lambda_3^p}''>0$, the least positive eigenvalue of $T[p,0]$ in \begin{equation*}V=\left\lbrace u\in C^4 (I)\quad\mid\quad u(a)=u''(a)=u(b)=u'(b)=0\right\rbrace \,.\end{equation*}
				
			\end{itemize}
		
		\end{itemize}

		\begin{remark}
				In \cite{CaSaa1} it is proved that the second order linear differential equation $u''(t)+m\,u(t)=0$ is disconjugate on $I$ if, and only if, $m\in\left( -\infty,\left( \frac{\pi}{b-a}\right)^2\right)$. In particular: if $p\geq 0$, then  $u''(t)-p\,u(t)=0$ is a disconjugate equation in every real interval $I$.
				
				Hence, under this disconjugacy condition, in \cite{CaSaa2} it is proved the existence of $\lambda_1^p>0$, ${\lambda_2^p}'<0$, ${\lambda_2^p}''<0$, ${\lambda_3^p}'>0$ and ${\lambda_3^p}''>0$. Thus, the previous eigenvalues are well-defined.
		\end{remark}
	
	As a consequence of \cite[Theorem 6.1]{CaSaa2} we can state the following result

	\begin{corollary}\label{Cor::1}
		We consider the operator $T[p,c]\,u(t)\equiv u^{(4)}(t)-p\,u''(t)+c(t)\,u(t)$, where $p\in\mathbb{R}$ and $p\geq 0$. Then,
		\begin{itemize}
			\item If $-\lambda_1^p<c(t)\leq -\lambda_2^p$ for every $t\in I$, then $T[p,c]$ is strongly inverse positive in $X$.
			\item If $-\lambda_3^p\leq c(t)<-\lambda_1^p$ for every $t\in I$, then $T[p,c]$ is strongly inverse negative in $X$.
		\end{itemize}
	\end{corollary}

Moreover, in \cite{CaSaa2}, there are obtained the values of $\lambda_1^p$, $\lambda_2^p$ and $\lambda_3^p$. In particular, we have:
	
	The eigenvalues of the operator $T[p,0]$ in $X$ are given by $\lambda_k=k^4\,\left( \frac{\pi}{b-a}\right) ^4+k^2\,p\,\left( \frac{\pi}{b-a}\right)^2$, where $k\in\{1,2,3,\dots\}$.
	
	 Obviously, the least positive eigenvalue is given by $\lambda_1^p=\left( \frac{\pi}{b-a}\right) ^4+p\,\left( \frac{\pi}{b-a}\right) ^2$. Moreover, we denote as  ${\lambda_1^p}'=16\,\left( \frac{\pi}{b-a}\right) ^4+4\,p\,\left( \frac{\pi}{b-a}\right) ^2$ the second positive  eigenvalue of $T[p,0]$ in $X$.
	
	It is clear that if we denote $\lambda$ as an eigenvalue of $T[p,0]$ and its associated eigenfunction as $u\in X_1$, then function $v(t):=u(1-t)$ is an eigenfunction associated to $\lambda$ in $X_3$. As a consequence, the eigenvalues of $T[p,0]$ on the spaces $X_1$ and $X_3$ are the same. So, in the previous definitions $\lambda_2^p={\lambda_2^p}'={\lambda_2^p}''$.
	
	One can verify that such eigenvalues are given as $-\lambda$, where $\lambda$ is a positive solution of
	\[\frac{ \tan \left(\dfrac{b-a}{2}\,\sqrt{2\sqrt{\lambda}-p} \right)}{\sqrt{2\sqrt{\lambda}-p}}=\frac{ \tanh \left(\dfrac{b-a}{2}\,\sqrt{2\sqrt{\lambda}+p} \right)}{\sqrt{2\sqrt{\lambda}+p}}\,,\]
	in particular, $\lambda_2^p$ is the opposite of the least positive solution of this equation.
	
	Similarly, the eigenvalues of $T[p,0]$ in $U$ and $V$ coincide and we conclude that $\lambda_3^p={\lambda_3^p}'={\lambda_3^p}''$. 
	
	In particular, the eigenvalues are given as the positive solutions of the following equality:
	\[\dfrac{\tan\left( \frac{(b-a)\,\sqrt{\sqrt{p^2+4\,\lambda}-p}}{\sqrt{2}}\right) }{\sqrt{\sqrt{p^2+4\,\lambda}-p}}=\dfrac{\tanh\left( \frac{(b-a)\,\sqrt{\sqrt{p^2+4\,\lambda}+p}}{\sqrt{2}}\right) }{\sqrt{\sqrt{p^2+4\,\lambda}+p}}\,,\]
and $\lambda_3^p$ is the least positive solution of this equation.
	
	\section{Relation between strongly inverse positive (negative) character of $T[p,c]$ in $X$ and $\tilde X$}
	
	In this section, we are going to stablish a relation between the strongly inverse positive (negative) character of the operator $T[p,c]$ on the set $\tilde X$, defined in \eqref{Ec::Xt}, and the strongly inverse positive (negative) character of $T[p,c]$ in $X$, defined in \eqref{Ec::esp}.

	First we introduce a previous result, which proof follows directly from the uniqueness of the homogeneous problem \eqref{Ec::Op},\eqref{Ec::cf}.
	
	\begin{lemma}\label{L::14}
		If the problem \eqref{Ec::Op},\eqref{Ec::cf} has only the trivial solution for $h\equiv 0$. Then \eqref{Ec::Op},\eqref{Ec::cfnh} has a unique solution given by:
		\begin{equation}
		\label{Ec::sol}
		u(t)=\int_a^bg_{p,c}(t,s)\,h(s)\,ds+d_1\,y_{p,c}^a(t)+d_2\,y_{p,c}^b(t), \quad t\in I,
		\end{equation}
		where $g_{p,c}(t,s)$ is the related Green's function of $T[p,c]$ in $X$ and:
		\begin{itemize}
			\item[] $y_{p,c}^a$ is defined as the unique solution of
			\begin{equation}
			\label{Ec::ya}\left\lbrace \begin{array}{c}
			T[p,c]\,u(t)=0\,,\quad t\in I,\\\\
			u''(a)=1\,,\quad u(a)=u(b)=u''(b)=0\,,\end{array} \right. 
			\end{equation}
			\item[] and $y_{p,c}^b$ is defined as the unique solution of
			\begin{equation}
			\label{Ec::yb}\left\lbrace \begin{array}{c}
			T[p,c]\,u(t)=0\,,\quad t\in I,\\\\
			u''(b)=1\,,\quad u(a)=u(b)=u''(a)=0\,.\end{array} \right. 
			\end{equation}	
		\end{itemize}
	\end{lemma}
	
	Now, we can prove the following result:
	
	\begin{theorem}\label{T::2.5}
		{\color{white}.}
		\begin{itemize}
			\item $T[p,c]$ is a strongly inverse positive operator in $\tilde X$ if, and only if, it is strongly inverse positive in $ X$.
			\item  $T[p,c]$ is a strongly inverse negative operator in $\tilde X$ if, and only if, it is strongly inverse negative in $X$.		
		\end{itemize}
	\end{theorem}
	
	\begin{proof}
		Since $X \subset \tilde X$, necessary condition is trivial.
		
		Now, let us see the sufficient one. From the strongly inverse positive (negative) character of $T[p,c]$ in $X$, using Theorem \ref{T::14}, we conclude that $g_{p,c}>0$ ($<0$) a.e. on $I\times I$. Then, we only need to study the sign of $y_{p,c}^a$ and $y_{p,c}^b$. 
		
		In order to do that, we are going to establish a relationship between these functions and some derivatives of $g_{p,c}(t,s)$.
		
		In \cite[Theorem 6.1]{CaSaa2}, it is obtained that  $w(t):=\dfrac{\partial}{\partial s}g_{p,c}(t,s)_{\mid s=a}$  satisfies:
		\begin{eqnarray}
		\nonumber T[p,c]\,w(t)&=&0\,,\quad \forall t\in(a,b]\,,\\\nonumber\\\nonumber
		w(a)=w(b)=w''(b)&=&0\,,\\\nonumber\\\nonumber
		w''(a)&=&-1\,,
		\end{eqnarray}
		thus, we deduce that $y_{p,c}^a(t)=-w(t)$ for all $t\in I$.
		
		Since $T[p,c]$ is a self-adjoint operator, $g_{p,c}(t,s)=g_{p,c}(s,t)$.
		
		Moreover, if $T[p,c]$ is strongly inverse positive (negative) in $X$, from Theorem \ref{T::14} and the symmetry of $g_{p,c}$, we have that $w>0$ ($<0$) a.e. on $I$. So, $y_{p,c}^a<0$ ($>0$) a.e. on $I$.

		Analogously, in \cite[Theorem 6.1]{CaSaa2} it is obtained that  $y(t):=\dfrac{\partial}{\partial s}g_{p,c}(t,s)_{\mid s=b}$ satisfies:
		\begin{eqnarray}
		\nonumber T[p,c]\,y(t)&=&0\,,\quad \forall t\in[a,b)\,,\\\nonumber\\\nonumber
		y(a)=y(b)=y''(a)&=&0\,,\\\nonumber\\\nonumber
		y''(b)&=&1\,.
		\end{eqnarray}
		
		Thus, we deduce that $y_{p,c}^b(t)=y(t)$ for all $t\in I$.
		
		Moreover, if $T[p,c]$ is strongly inverse positive (negative) in $X$, from Theorem \ref{T::14} and the symmetry of $g_{p,c}$, we have that $y<0$ ($>0$) a. e. on $I$. So, $y_{p,c}^b<0$ ($>0$) a.e. on $I$ too.
		
		Hence, the result is proved.
	\end{proof}
	
	So, we have proved that the strongly inverse positive (negative) character of $T[p,c]$ in $X$ and $\tilde X$ are equivalent. So, if we are able to prove that $T[p,c]$ is either strongly inverse positive or strongly inverse negative in one of these two spaces, then such property is also fulfilled  in the other one.
	
	In the sequel, we are going to obtain some sufficient conditions to ensure that $T[p,c]$ is strongly inverse positive (negative) in $X$ and $\tilde X$. From Theorem \ref{T::2}, it is enough to prove it for $X$.
	
	\section{Variational approach}
	
	In this section we are going to obtain the variational approach of problem \eqref{Ec::Op},\eqref{Ec::cf} and some results which will be used on our main results.
	
	 First, we consider the Hilbert space $H:=H^2(I)\cap H_0^1(I)$, where:
	 \[H^2(I)=\{u\in  L^2(I)\ \mid u'\,,u''\in  L^2(I)\}\,,\]
	 	and\[H_0^1(I)=\{u\in  L^2(I)\ \mid u'\in  L^2(I)\,,\ u(a)=u(b)=0\}\,.\]
	 
	  We say that $u\in H$ is a weak solution of \eqref{Ec::Op},\eqref{Ec::cf} if it satisfies
{\footnotesize	\begin{equation}\label{Ec::fvar}\int_a^bu''(t)\,v''(t)\,dt+p\int_a^bu'(t)\,v'(t)\,dt+\int_a^bc(t)\,u(t)\,v(t)\,dt=\int_a^bh(t)\,v(t)\,dt\,,\quad \forall v\in H\,.\end{equation}}

	For a function $f\in C(I)$. Let us denote
	\[f_m:=\min_{t\in I}f(t)\quad\quad\text{and}\quad\quad f^m:=\max_{t\in I}f(t)\,,\]
	and
	\[f^{\pm}(t)=\max\left\lbrace 0,\pm f(t)\right\rbrace\,,\ t\in I. \]
	
	If $p=0$ and $a=0$, $b=1$, we have the following result, see \cite{Usm, Yang}.
	
	\begin{proposition}\label{P::Usm}
		Let $c(t)\neq -k^4\,\pi ^4$ for any $k\in\mathbb{N}$ and all $t\in[0,1]$. Let $p=0$, $a=0$ and $b=1$, then the problem \eqref{Ec::Op},\eqref{Ec::cf} has a unique solution $u\in X$. Moreover, if $-\pi^4<c_m<0$, then
		\[\left\| u\right\|_{C([0,1])}\leq \dfrac{\pi}{2\,(\pi^4+c_m)}\,\left\| h\right\|_{C([0,1])}\,.\]
	\end{proposition}
	
	Now, we are going to enunciate an equivalent result to this Proposition, which refers to our case.
	
	\begin{proposition}\label{P::Usm1}
		
				Let $c(t)\neq -k^4\,\left( \frac{\pi}{b-a}\right)  ^4-k^2\,p\,\left( \frac{\pi}{b-a}\right)^2 $ for any $k\in\{1,2,3,\dots\}$ and all $t\in I$. Then the problem \eqref{Ec::Op},\eqref{Ec::cf} has a unique solution $u\in X$.
				
				 Moreover, if $-\left( \frac{\pi}{b-a}\right)  ^4-p\,\left( \frac{\pi}{b-a}\right)^2<c_m<0$, then
				\[\left\| u\right\|_{C([0,1])}\leq \dfrac{\pi}{2\,\left( \left( \frac{\pi}{b-a}\right)  ^4+p\,\left( \frac{\pi}{b-a}\right)^2+c_m\right) }\,\left\| h\right\|_{C([0,1])}\,.\]

	\end{proposition}
	
	\begin{proof}
		If $c(t)\neq -k^4\,\left( \frac{\pi}{b-a}\right)  ^4-p\,k^2\left( \frac{\pi}{b-a}\right)^2 $ for any $k\in\{1,2,3,\dots\}$ and $t\in I$, it means that, since $c\in C( I)$, either there exist $k\in \{1,2,,3\dots\}$ such that
	{\scriptsize 	$$c(t)\in\left( -(k+1)^4\left( \frac{\pi}{b-a}\right)  ^4-p\,(k+1)^2\left( \frac{\pi}{b-a}\right)^2,-k^4\left( \frac{\pi}{b-a}\right)  ^4-p\,k^2\left( \frac{\pi}{b-a}\right)^2\right) $$} \hspace{-0.13cm}or that $c_m>-\left( \frac{\pi}{b-a}\right)  ^4-p\,\left( \frac{\pi}{b-a}\right)^2$, i.e.  there is no any eigenvalue of $T[0,p]$ between $c_m$ and $c^m$. As a consequence, the existence of a unique solution of problem \eqref{Ec::Op}, \eqref{Ec::cf} is ensured. Now, let us see the boundedness.
		
			We have the two following Wirtinger inequalities for every $u\in H$, (see \cite{PTV, Usm})
			\begin{equation}\label{Ec::Des1}\left\| u\right\| _{ L^2(I)}\leq \dfrac{b-a}{\pi}\left\|u' \right\|_{ L^2(I)}\leq \left( \dfrac{b-a}{\pi}\right) ^2\left\| u''\right\|_{ L^2(I)}\,,\end{equation}
			and
			\begin{equation}\label{Ec::Des2}\left\| u\right\| _{C( I)}\leq \dfrac{\sqrt{b-a}}{2}\left\| u'\right\| _{ L^2(I)}\,.\end{equation}
			
			Now, multiplying  equation \eqref{Ec::Op} by the unique solution $u\in X$ and  integrating, we have
			\[\int_a^bu^{(4)}(t)\,u(t)\,dt-p\,\int_a^bu''(t)\,u(t)\,dt+\int_a^bc(t)\,u^2(t)\,dt=\int_a^bh(t)\,u(t)\,dt\,,\]
			which is equivalent to
			\[\int_a^b{u''}^2(t)\,dt+p\int_a^b{u'}^2(t)\,dt=\int_a^bh(t)\,u(t)\,dt-\int_a^bc(t)\,u^2(t)\,dt\,.\]
			
			Now, taking into account the inequalities \eqref{Ec::Des1} and that $c_m\leq 0$  we have
		{\scriptsize 	\begin{eqnarray*}
			 \left\|u''\right\|^2_{ L^2(I)}+ p\left\| u'\right\|^2_{ L^2(I)} &\geq&\left( \dfrac{\pi}{b-a}\right) ^2 \left\| u'\right\|^2_{ L^2(I)}+p\left\| u'\right\|^2_{ L^2(I)}
			\end{eqnarray*}}
			and
			
			{\scriptsize 	\begin{eqnarray*}
			\int_a^bh(t)\,u(t)\,dt-\int_a^bc(t)\,u^2(t)\,dt&\leq&\left\|h\right\| _{C( I)}   \int_a^b\,\left|u(t)\right|\,dt  -c_m\left\| u\right\|^2_{ L^2(I)} \\\nonumber
			&\leq &\left\|h\right\| _{C( I)} \sqrt{b-a} \left\| u\right\|_{ L^2(I)}  -c_m\left\| u\right\|^2_{ L^2(I)}\\\nonumber
			&\leq& \left\|h\right\| _{C( I)} \sqrt{b-a} \dfrac{b-a}{\pi}\left\| u'\right\|_{ L^2(I)} -c_m\left( \dfrac{b-a}{\pi}\right)^2 \left\| u'\right\|^2_{ L^2(I)}
			\end{eqnarray*}}
		
		So, combining the last two inequalities we arrive to
		
		\[\left( \left( \dfrac{\pi}{b-a}\right) ^2+p+ c_m\left( \dfrac{b-a}{\pi}\right)^2\right) \left\| u'\right\|_{ L^2(I)}\leq \left\|h\right\| _{C( I)} \sqrt{b-a} \dfrac{b-a}{\pi}\,, \]
		which is equivalent to
			\[\left\| u'\right\|_{L_2(I)}\leq \dfrac{\pi}{\sqrt{b-a}}\dfrac{\left\| h\right\|_{C(I)}}{\left( \left( \frac{\pi}{b-a}\right)  ^4+p\,\left( \frac{\pi}{b-a}\right)^2+c_m\right) }\,,\]
			that combined with the inequality \eqref{Ec::Des2} gives our result.
		\end{proof}
		
		\begin{remark}
			We note that previous inequality includes Proposition \ref{P::Usm} as a particular case.
		\end{remark}
		
			For an arbitrary nonnegative continuous function $r(t)\geq 0$ in $ I$, we define the scalar product
			\begin{equation}\label{Ec::pr-es}(u,v)=\int_a^bu''(t)\,v''(t)\,dt+p\int_a^bu'(t)\,v'(t)\,dt+\int_a^br(t)\,u(t)\,v(t)\,dt\,,\quad u\,,v\in H\,,\end{equation}
			and $\left\| u\right\| =(u,u)^{1/2}$ its associated norm.
			
				We have the following inequality:
				\begin{eqnarray}\nonumber\left| u(t)-u(s)\right|& =&\left| \int_s^tu'(r)\,dr\right| \leq \sqrt{t-s}\,\left\| u'\right\| _{L^2(I)}\leq \sqrt{{t-s}}\,\frac{b-a}{\pi}\left\| u''\right\|_{L^2(I)}\\\nonumber &\leq& \sqrt{{t-s}}\,\frac{b-a}{\pi}\left\| u\right\|\,.\end{eqnarray}

					Thus, we can affirm that the embedding of $H$ into $C(I)$ is compact.

		Let $f(t)$ and $h(t)$ be continuous functions on $ I$, following the arguments shown in \cite{Dra}, using  the Riesz Representation  Theorem we can define $S_f\colon H\rightarrow H$ and $h^*\in H$ such that 
			\begin{equation}\label{Ec::Sf-h}(S_fu,v)=\int_a^bf(t)\,u(t)\,v(t)\,dt\,,\quad (h^*,v)=\int_a^b h(t)\,v(t)\,dt\,,\ u,v\in H\,.\end{equation}

			Now, let us introduce some results which make a relation between this norm and the norms $\left\| \cdot\right\| _{C( I)}$ and $\left\| \cdot\right\| _{ L^2(I)}$. Such result generalizes \cite[Lemma 7]{Dra}.
			
			\begin{lemma} \label{L::1}
				Let $u\in H$, $r \in C(I)$, $r\geq 0$ in $I$ and $\|\cdot\|$ be the norm associated to the scalar product \eqref{Ec::pr-es}. Then
				\begin{eqnarray}
				\nonumber\left\| u\right\| _{C( I)}&\leq &\dfrac{1}{\sqrt{\delta_1}}\left\| u\right\| \,,\\\nonumber
				\end{eqnarray}	
				and
				\begin{eqnarray}
			\nonumber	\left\| u\right\| _{ L^2(I)}&\leq & \dfrac{\left\| u\right\| }{\sqrt{\left( \dfrac{\pi}{b-a}\right) ^4+p\,\left( \dfrac{\pi}{b-a}\right) ^2
			+\min_{t\in I}{\{r(t)\}}}}\,,
				\end{eqnarray}
					where \begin{equation}
					\label{Ec::D1}
				\delta_1=\max\left\lbrace \dfrac{4\,p}{b-a},\dfrac{4\,\pi^2}{(b-a)^{3/2}}\right\rbrace\,. \end{equation}

			\end{lemma}
			
			\begin{proof}

			Using the inequalities given in \eqref{Ec::Des1}-\eqref{Ec::Des2}, we have that the two following inequalities are satisfied
				
				\begin{eqnarray}
				\nonumber	p\,\left\| u\right\| ^2_{C( I)}&\leq& \dfrac{b-a}{4}\,p\int_a^b(u'(t))^2\,dt\\\nonumber&\leq& \dfrac{b-a}{4}\left( \int_a^b(u''(t))^2\,dt+p\int_a^b(u'(t))^2\,dt+\int_a^br(t)\,u^2(t)\,dt\right) \\\nonumber&=&\dfrac{b-a}{4}\left\| u\right\| ^2\,,\\\nonumber&&\\\nonumber&&\\\nonumber 	\left\| u\right\| ^2_{C( I)}&\leq& \dfrac{b-a}{4}\int_a^b(u'(t))^2\,dt\leq \dfrac{(b-a)^3}{4\,\pi^2}\int_a^b(u''(t))^2\,dt\\\nonumber&\leq& \dfrac{(b-a)^3}{4\,\pi^2}\left( \int_a^b(u''(t))^2\,dt+p\int_a^b(u'(t))^2\,dt+\int_a^br(t)\,u^2(t)\,dt\right) \\\nonumber&=&\dfrac{(b-a)^3}{4\,\pi^2}\left\| u\right\| ^2\,.
				\end{eqnarray}
				
				So, if $p\neq 0$,
				\[\left\| u\right\| _{C( I)}\leq \min\left\lbrace \sqrt{\dfrac{b-a}{4\,p}},\dfrac{\sqrt{(b-a)^3}}{2\,\pi}\right\rbrace \left\|u\right\|  =\dfrac{1}{\sqrt{\delta_1}} \left\|u\right\| \,,\]				
			moreover, if $p=0$,				
					\[\left\| u\right\| _{C( I)}\leq  \dfrac{\sqrt{(b-a)^3}}{2\,\pi} \left\|u\right\|  =\dfrac{1}{\sqrt{\delta_1}} \left\|u\right\| \,.\]
					
				On another hand,
				\begin{eqnarray}
				\nonumber\left\| u\right\| ^2_{ L^2(I)}&=&\dfrac{\left( \dfrac{\pi}{b-a}\right) ^4+p\,\left( \dfrac{\pi}{b-a}\right) ^2+\min_{t\in I}{\{r(t)\}}}{\left( \dfrac{\pi}{b-a}\right) ^4+p\,\left( \dfrac{\pi}{b-a}\right) ^2+\min_{t\in I}{\{r(t)\}}}\int_a^bu^2(t)\,dt\\\nonumber&&\\\nonumber
				&\leq&  
				\dfrac{\int_a^b(u''(t))^2\,dt+p\,\int_a^b(u'(t))^2\,dt+\int_a^br(t)\,u^2(t)\,dt}{\left( \dfrac{\pi}{b-a}\right) ^4+p\,\left( \dfrac{\pi}{b-a}\right) ^2+\min_{t\in I}{\{r(t)\}}}\\\nonumber&&\\\nonumber
				&=&\dfrac{\left\| u\right\| ^2}{\left( \dfrac{\pi}{b-a}\right) ^4+p\,\left( \dfrac{\pi}{b-a}\right) ^2+\min_{t\in I}{\{r(t)\}}}\,.
				\end{eqnarray}
			\end{proof}

From classical arguments, see \cite{Bre}, we obtain the following result, where we see that a  weak solution of \eqref{Ec::fvar} in $H$ under suitable conditions is indeed a classical solution of \eqref{Ec::Op}-\eqref{Ec::cf} in $X$.

\begin{proposition} \label{P::4.5}
	If $c,h\in C( I)$, then if $u\in H$ is a weak solution of \eqref{Ec::fvar}, then $u$ is a classical solution of \eqref{Ec::Op}-\eqref{Ec::cf} in $X$.
\end{proposition}

%
%
%
%

				Next result improves  \cite[Lemma 8]{Dra}
			\begin{lemma}\label{L::2}
				Let $S_f\colon H\rightarrow H$ be the operator previously defined in \eqref{Ec::Sf-h}. Then,
				\[\left\| S_f\right\| \leq \dfrac{1}{\delta_1}\int_a^b\,|f(t)|\,dt\,.\]
			\end{lemma}
			
			\begin{proof}
				
				Using Lemma \ref{L::1} we can deduce the following inequalities which prove the result:
				\begin{eqnarray}
				\nonumber
				\left\| S_f\right\| &=&\sup_{\left\| u\right\| =1}\left\| S_f\,u\right\| =\sup_{\left\| u\right\| =1}\sup_{\left\| v\right\| =1}\left| \int_a^bf(t)\,u(t)\,v(t)\,dt\right| \\\nonumber
				&\leq& \sup_{\left\| u\right\| =1}\sup_{\left\| v\right\| =1}\int_a^b|f(t)|\,|u(t)|\,|v(t)|\,dt\\\nonumber&\leq& \sup_{\left\| u\right\| =1}\left\| u\right\| _{C( I)} \sup_{\left\| v\right\| =1}\left\| v\right\| _{C( I)}\int_a^b|f(t)|\,dt\\\nonumber
				&\leq& \dfrac{1}{\delta_1}\int_a^b|f(t)|\,dt\,.
				\end{eqnarray}
				
			\end{proof}
			
			Repeating the previous argument, we have
				\[\left\| S_f(u_n-u_m)\right\| \leq  \dfrac{1}{\sqrt{\delta_1}}\int_a^b\left| f(t)\right| \,dt \,\left\|u_n-u_m\right\|_{C(I)}\,. \]
			
			Thus, from the compact embedding of $H$ into $C(I)$, we can affirm that $S_f:H \to H$ is a compact operator.
			
			The proof of next result  is analogous to \cite[Lemma 9]{Dra}.
			
			\begin{lemma}\label{L::3}
				Let $h^*\in H$ previously defined in \eqref{Ec::Sf-h}. Then
				\[\left\| h^*\right\| \leq \sqrt{\dfrac{b-a}{\left( \dfrac{\pi}{b-a}\right) ^4+p\,\left( \dfrac{\pi}{b-a}\right) ^2+\min_{t\in I}{\{r(t)\}}}}\,\left\| h\right\|_{C( I)}\,.\]
			\end{lemma}
			
%
		
	\section{Strongly inverse positive (negative) character of $T[p,c]$ in $\tilde X$.}

This section is devoted to prove maximum and anti-maximum principles for the problem \eqref{Ec::Op},\eqref{Ec::cfnh}. These results generalize those obtained in \cite{Dra, Dra1} for $p=0$ and the homogeneous boundary conditions. The proofs follow similar arguments to the ones given in such articles. We point out that on them there is no reference to spectral theory. 

First, we obtain the results for the homogeneous case and, once  we have done that, we extrapolate them for the non-homogeneous boundary conditions by means of Lemma \ref{L::14} and Theorem \ref{T::2.5}.

The first of them ensures the existence of a unique solution of the problem under certain hypothesis and gives sufficient conditions to ensure that the operator \eqref{Ec::T4} is strongly inverse positive in $X$.
	
	\begin{theorem}\label{T::2}
		Let $c\,,h\in C( I)$ be such that 
		\[\int_a^bc^-(t)\,dt<\delta_1\,,\]
		where $\delta_1$ has been defined in \eqref{Ec::D1}. Then Problem \eqref{Ec::Op},\eqref{Ec::cf} has a unique classical solution $u\in X$ and there exists $R>0$ (depending on $c$ and $p$) such that 
		\[\left\| u\right\| _{C( I)}\leq R\,\left\| h\right\|_{C( I)}\,.\]
		
		Moreover, if $c(t)\leq -\lambda_2^p$, for every $t\in I$, then $T[p,c]$ is strongly inverse positive in $ X$.
		
	\end{theorem}
	\begin{proof}

			First, we decompose $c(t)=c^+(t)-c^-(t)$. And, we write the problem \eqref{Ec::Op},\eqref{Ec::cf} as follows
			\begin{eqnarray}
			\nonumber u^{(4)}(t)-p\,u''(t)+c^+(t)\,u(t)&=&c^-(t)\,u(t)+h(t)\,,\quad t\in I,\\\nonumber
			u(a)=u(b)=u''(a)=u''(b)&=&0\,.
			\end{eqnarray} 
			
			If we denote $r(t):=c^+(t)$ and $f(t):=c^-(t)$, we have that the weak formulation of problem \eqref{Ec::Op} is given in the following way
			\begin{equation}\label{E::1}
			u=S_{c^-}\,u+h^*\,,\quad u\in H\end{equation}
			with the scalar product $(\cdot,\cdot)$ previously defined in \eqref{Ec::pr-es}.
			
			Using Lemma \ref{L::2} we have
			\begin{eqnarray}
			\nonumber \left\| S_{c^-}\right\| &\leq& \dfrac{1}{\delta_1}\,\int_a^b\left| c^-(t)\right| \,dt=\dfrac{1}{\delta_1}\,\int_a^b c^-(t) \,dt<\dfrac{1}{\delta_1}\,\delta_1=1\,.
			\end{eqnarray}
			
			Hence, $ S_{c^-}$ is a contractive operator and there exists a unique weak solution $u\in H$. From Proposition \ref{P::4.5}, $u\in X$ is a classical solution of \eqref{Ec::Op} in $X$.
			
			Now, using \eqref{E::1} we obtain:
			
			\[\left\| u\right\| =\left\| S_{c^-}\,u+h^*\right\| \leq \left\| S_{c^-}\right\| \,\left\| u\right\| +\left\| h^*\right\| \,,\]
			then
			\[\left\| u\right\| \leq\dfrac{1}{1-\left\| S_{c^-}\right\| }\,\left\| h^*\right\| \,.\]
			
			By another hand, using Lemmas \ref{L::1} and \ref{L::3}
			{\footnotesize 	\begin{eqnarray}\nonumber
				\left\| u\right\|_{C( I)}&\leq& \dfrac{1}{\sqrt{\delta_1}}\,\left\| u\right\| \leq \dfrac{1}{\sqrt{\delta_1}(1-\left\| S_{c^-}\right\| )}\,\left\| h^*\right\|\\\nonumber
				&\leq&\dfrac{\sqrt{b-a}}{\sqrt{\delta_1}(1-\left\| S_{c^-}\right\| )\sqrt{\left( \dfrac{\pi}{b-a}\right)^4+p\,\left( \dfrac{\pi}{b-a}\right) ^2+
				\min_{t\in I}{\{c^+(t)\}}} }\,\left\| h\right\|_{C( I)}\\
				&=: & R\, \,\left\| h\right\|_{C( I)}.\end{eqnarray}}
			
			Moreover, from Lemma \ref{L::2}, we know that
			\begin{eqnarray}
			\nonumber
			R \le \dfrac{\sqrt{b-a}}{\dfrac{1}{\sqrt{\delta_1}}\left( \delta_1-\int_a^bc^-(t)\,dt\right) \sqrt{\left( \dfrac{\pi}{b-a}\right)^4+p\,\left( \dfrac{\pi}{b-a}\right) ^2+\min_{t\in I}{\{c^+(t)\}}} }\,.
			\end{eqnarray}

			The proof behind here is just the same as in the particular case of $p=0$, which is collected in \cite[Theorem 4]{Dra}.
	\end{proof}

Now, we introduce a result which also gives us sufficient conditions to ensure the existence of solution of our problem and, moreover, it warrants that the operator $T[p,c]$ is strongly inverse negative in $X$.
	\begin{theorem}\label{T::3}
		Let $c\,, h\in C( I)$ be such that $-16 \left( \frac{\pi}{b-a}\right)^4-4\,p\,\left( \frac{\pi}{b-a}\right)^2  <c_m<- \left( \frac{\pi}{b-a}\right)^4-p\,\left( \frac{\pi}{b-a}\right)^2 $ and
		\[\int_a^b\left( c(t)-c_m\right) \,dt<\delta_1\,\delta_2\,,\]
		where $\delta_1$ is defined on \eqref{Ec::D1} and {\footnotesize \[\delta_2=\min \left\lbrace -1-\dfrac{c_m}{ \left( \frac{\pi}{b-a}\right)^4+p\,\left( \frac{\pi}{b-a}\right)^2},1+\dfrac{c_m}{16 \left( \frac{\pi}{b-a}\right)^4+4p\,\left( \frac{\pi}{b-a}\right)^2 }\right\rbrace\,.\]}
		
		Then problem \eqref{Ec::Op},\eqref{Ec::cf}  has a unique classical solution on $X$ and there exists $R>0$ (depending on $c$ and $p$) such that
		\[\left\| u\right\|_{C( I)}\leq R\,\left\| h\right\| _{C( I)}\,.\]
		
		Moreover, if $c_m\geq -\lambda^p_3$, then $T[p,c]$ is strongly inverse negative in $X$.
	\end{theorem}
	
	\begin{proof}
			We rewrite the problem \eqref{Ec::Op},\eqref{Ec::cf} in the following way
			\begin{eqnarray}
			u^{(4)}(t)-p\,u''(t)+c_m\,u(t)&=&(c_m-c(t))\,u(t)+h(t)\,,\quad t\in I\,,\\\nonumber
			u(a)=u(b)=u''(a)=u''(b)&=&0\,.
			\end{eqnarray}
			
			In this case, we consider $r(t)\equiv 0$ and we have that the weak formulation is equivalent to 
			\begin{equation}
			\label{Ec::fd} u+S_{c_m}\,u=S_{c_m-c}\,u+h^*\,.
			\end{equation}
			
			Since {\footnotesize $c_m\in\left( -16 \left( \frac{\pi}{b-a}\right)^4-4\,p\,\left( \frac{\pi}{b-a}\right)^2 ,- \left( \frac{\pi}{b-a}\right)^4-p\,\left( \frac{\pi}{b-a}\right)^2\right) $}, we have that $I+S_{c_m}$ is invertible in $H$ . Then we can write
			\begin{equation}\label{Ec:fd2}
			u=(I+S_{c_m})^{-1}\,(S_{c_m-c}\,u+h^*)\,.
			\end{equation}
			
			Moreover $\left\| (I+S_{c_m})^{-1}\right\| =\dfrac{1}{\delta_2}$, where $\delta_2$ is the distance from $-1$ to the spectrum of $S_{c_m}$, see \cite{Sad}, i.e. 
			{\footnotesize \[\delta_2=\min \left\lbrace -1-\dfrac{c_m}{ \left( \frac{\pi}{b-a}\right)^4+p\,\left( \frac{\pi}{b-a}\right)^2},1+\dfrac{c_m}{16 \left( \frac{\pi}{b-a}\right)^4+4\,p\,\left( \frac{\pi}{b-a}\right)^2 }\right\rbrace \,.\]}
			
			Since $\int_a^b(c(t)-c_m)\,dt<\delta_1\,\delta_2$ and using Lemma \ref{L::2}, we have
			\[\left\| S_{c_m-c}\right\| \leq \dfrac{1}{\delta_1}\,\int_a^b(c(t)-c_m)\,dt<\dfrac{\delta_1\,\delta_2}{\delta_1}=\delta_2\,,\]
			so,
			\[\left\| (I+S_{c_m})^{-1}\,S_{c_m-c}\right\| \leq \left\| (I+S_{c_m})^{-1}\right\| \,\left\| S_{c_m-c}\right\| <\dfrac{\delta_2}{\delta_2}=1\,.\]
			
			So, analogously to Theorem \ref{T::2}, we can use the contractive character of operator $(I+S_{c_m})^{-1}\,S_{c_m-c}$, to ensure that there exists a unique weak solution  of \eqref{Ec::Op},\eqref{Ec::cf} $u\in H$, from Proposition \ref{P::4.5}, it is also a classical solution $u\in X$.
			
			Now, using \eqref{Ec:fd2}, we have
			\[\left\| u\right\| \leq \left\| (I+S_{c_m})^{-1}\right\| \,\left\| S_{c_m-c}\,u+h\right\|=\dfrac{1}{\delta_2}\left\| S_{c_m-c}\right\| \,\left\| u\right\| +\dfrac{1}{\delta_2}\left\| h^*\right\|\,.\]
			
			As consequence, we deduce that
			\[\left\| u\right\| \leq \dfrac{\left\| h^*\right\| }{\delta_2-\left\| S_{c_m-c}\right\| }\,,\]
			so, combining this inequality with Lemmas \ref{L::1} and \ref{L::3}, we have
			
			\begin{eqnarray}\nonumber\left\| u\right\| _{C( I)}&\leq& \dfrac{1 }{\sqrt{\delta_1}}\left\| u\right\| \leq\dfrac{\left\| h^*\right\| }{\sqrt{\delta_1}(\delta_2-\left\| S_{c_m-c}\right\|) }\\\nonumber&&\\\nonumber
			&\leq& \dfrac{1 }{\sqrt{\delta_1}(\delta_2-\left\| S_{c_m-c}\right\|)}\sqrt{\dfrac{b-a}{\left( \dfrac{\pi}{b-a}\right) ^4+p\left( \dfrac{\pi}{b-a}\right) ^2}}\,\left\| h\right\| _{C( I)}\\
		&:=&	  R \, \,\left\| h\right\| _{C( I)}	\,.\end{eqnarray}
			
			Moreover 
			\begin{eqnarray}
			\nonumber 
			R \le  \dfrac{1 }{\frac{1}{\sqrt{\delta_1}}(\delta_1\delta_2-\int_a^b(c(t)-c_m)\,dt)}\sqrt{\dfrac{b-a}{\left( \dfrac{\pi}{b-a}\right) ^4+p\left( \dfrac{\pi}{b-a}\right) ^2}}\,.
			\end{eqnarray}
			
			The proof of the fact that  while $-\lambda^p_3\leq c_m\leq- \left( \frac{\pi}{b-a}\right)^4-\left( \frac{\pi}{b-a}\right)^2$, $T[p,c]$ is strongly inverse negative is equal to \cite[Theorem 5]{Dra}.
	\end{proof}
	
		From Lemma \ref{L::14} and Theorem \ref{T::2.5}, we can now state two results for the non-homogeneous case.
		
		\begin{theorem}\label{T::2nh}
			Let $c\,,h\in C( I)$ be such that 
			\[\int_a^bc^-(t)\,dt<\delta_1\,,\]
			where $\delta_1$ has been defined in \eqref{Ec::D1}. Then Problem \eqref{Ec::Op},\eqref{Ec::cfnh} has a unique classical solution $u\in \tilde X$.
			
			Moreover, if $c(t)\leq -\lambda_2^p$, for every $t\in I$, then $T[p,c]$ is strongly inverse positive in $ \tilde X$.
		\end{theorem}
		
		\begin{theorem}\label{T::3nh}
			Let $c\,, h\in C( I)$ be such that $-16 \left( \frac{\pi}{b-a}\right)^4-4\,p\,\left( \frac{\pi}{b-a}\right)^2  <c_m<- \left( \frac{\pi}{b-a}\right)^4-p\,\left( \frac{\pi}{b-a}\right)^2 $ and
			\[\int_a^b\left( c(t)-c_m\right) \,dt<\delta_1\,\delta_2\,,\]
			where $\delta_1$ is defined on \eqref{Ec::D1} and {\footnotesize \[\delta_2=\min \left\lbrace -1-\dfrac{c_m}{ \left( \frac{\pi}{b-a}\right)^4+p\,\left( \frac{\pi}{b-a}\right)^2},1+\dfrac{c_m}{16 \left( \frac{\pi}{b-a}\right)^4+4p\,\left( \frac{\pi}{b-a}\right)^2 }\right\rbrace\,.\]}
			
			Then problem \eqref{Ec::Op},\eqref{Ec::cfnh}  has a unique classical solution on $\tilde X$.
			
			Moreover, if $c_m\geq -\lambda^p_3$, then $T[p,c]$ is strongly inverse negative in $ \tilde X$.
		\end{theorem}
		
		\section{Maximum and anti-maximum principles for problem \eqref{Ec::Op},\eqref{Ec::cf} with $h>0$}
	In this section, even though we are not able to ensure the strongly inverse positive character of the operator $T[p,c]$ on $X$, we can ensure that, under the hypothesis that $h(t)>0$ for every $t\in I$, then problem \eqref{Ec::Op},\eqref{Ec::cf} has a unique positive (resp. negative) solution in $X$. The proofs follow similar steps to the ones given in \cite{Dra1}.
	
	\begin{theorem}\label{T::4}
		Let $h\in C( I)$ be a function such that $0<h_m\leq h^m$. Let $c\in C ( I)$ be a function  that   satisfies one of the two following hypothesis:
		\begin{enumerate}
			\item[(1)] {\footnotesize  $- \left( \frac{\pi}{b-a}\right)^4-p\,\left( \frac{\pi}{b-a}\right)^2<c_m\leq 0$} and {\footnotesize $c^m\leq -\lambda_2^p+\frac{h_m}{h^m}\,\frac{2}{\pi}\,\left( \left( \frac{\pi}{b-a}\right)^4+p\,\left( \frac{\pi}{b-a}\right)^2+c_m\right) $. }
			\item[(2)]  $ \int_a^bc^-(t)\,dt<\delta_1$, with $\delta_1$ defined on Theorem \ref{T::2}, and
		{\scriptsize 	\[c^m\leq -\lambda_2^p+\frac{h_m}{h^m}\frac{1}{\sqrt{\delta_1}}\left( \delta_1-\int_a^bc^-(t)\,dt\right) \sqrt{ \left( \frac{\pi}{b-a}\right) ^4+p\left( \frac{\pi}{b-a}\right) ^2+\min_{t\in I}\left\lbrace c^+(t),-\lambda_2^p\right\rbrace}\,,\]}
		\end{enumerate}
		then  problem \eqref{Ec::Op},\eqref{Ec::cf} has a unique positive solution in $X$.
	\end{theorem}
\begin{proof}
	The existence of a unique solution  follows from Proposition \ref{P::Usm1} on the first case and from Theorem \ref{T::2} on the second one.
	
	Now, let us see that this solution $u$ is positive on $(a,b)$.
	
	Let us assume that $c^m>-\lambda_2^p$. If $c^m\leq -\lambda_2^p$,   we can apply Corollary \ref{Cor::1} or Theorem \ref{T::2}, respectively, to affirm that $T[p,c]$ is strongly inverse positive on $X$.
	
	Let $d(t):=\min\left\lbrace c(t),-\lambda^p_2\right\rbrace $ be a continuous function such that $c_m\leq d(t)\leq -\lambda_2^p$. We transform the equation \eqref{Ec::Op} in the following equivalent one:
	\[T[p,d]\,u(t)=h(t)-\left( c(t)-d(t)\right) \,u(t)\,,\]
	and we consider the next recurrence formula
	\[T[p,d]\,u_{n+1}=h(t)-\left( c(t)-d(t)\right) \,u_n\,,\quad n=0,1,2,\dots\]

	Since, $d(t)\leq -\lambda_2^p$ for every $t\in I$, $T[p,d]$ is a strongly inverse positive operator in $X$.
	
	We choose $u_0\equiv 0$ and we have $T[p,d]\,u_{1}=h(t)$. Since $T[p,d]$ is strongly inverse positive, $u_1>0$ in $(a,b)$, $u_1'(a)>0$ and $u_1'(b)<0$.
	
	Now, using that $u_1 \in X$ is the unique solution of $T[p,d]\,u(t)=h(t)$, we deduce that 
	
	\begin{itemize}
		\item If $c$ satisfies $(1)$, we can apply Proposition \ref{P::Usm1} to affirm	\[\left\| u_1\right\|_{C([0,1])}\leq \dfrac{\pi}{2\,\left( \left( \frac{\pi}{b-a}\right)  ^4+p\,\left( \frac{\pi}{b-a}\right)^2+c_m\right) }\,h^m\,.\]
		\item If $c$ satisfies $(2)$, we look at the proof of the Theorem \ref{T::2} to conclude that
	{\scriptsize 	\[\left\| u_1\right\| _{C( I)}\leq  \dfrac{h^m}{\dfrac{1}{\sqrt{\delta_1}}\left( \delta_1-\int_a^bc^-(t)\,dt\right) \sqrt{\left( \dfrac{\pi}{b-a}\right)^4+p\,\left( \dfrac{\pi}{b-a}\right) ^2+\min_{t\in I}\left\lbrace c^+(t),-\lambda_2^p\right\rbrace } }\,.\]}
		
	\end{itemize}
	
	Since $c(t)-d(t)\leq c^m+\lambda_2^p$, in both  cases, using the hypothesis, we have
{\small 	
	\[T[p,d]\,u_2=h(t)-\left( c(t)-d(t)\right) \,u_1\geq h_m-\left( c^m+\lambda_2^p\right) \,\left\| u_1\right\| _{C( I)}\geq h_m-\dfrac{h_m}{h^m}\dfrac{1}{R}\,h^m\,R=0\,,\]}
	where $R$ is defined by

$$R:= \dfrac{\pi}{2\,\left( \left( \frac{\pi}{b-a}\right)  ^4+\left( \frac{\pi}{b-a}\right)^2+c_m\right) }$$
 if $(1)$ holds, and 
$$R:=  \dfrac{1}{\frac{1}{\sqrt{\delta_1}}\left( \delta_1-\int_a^bc^-(t)\,dt\right)\sqrt{\left( \dfrac{\pi}{b-a}\right)^4+p\,\left( \dfrac{\pi}{b-a}\right) ^2+\min_{t\in I}\left\lbrace c^+(t),-\lambda_2^p\right\rbrace } },$$ when $(2)$ is fulfilled.

	From here, the proof is equal than in the case that $p=0$, see \cite[Theorem 4]{Dra1}
\end{proof}
\begin{theorem}\label{T::5}
	Let $h\in C( I)$ be a function such that $0<h_m\leq h^m$. Let $c\in C ( I)$ be a function that satisfies
		\[\int_a^b\left( c(t)-c_m\right) \,dt<\delta_1\,\delta_2\,,\]
		where $\delta_1$ and $\delta_2$ have been defined in Theorems \ref{T::2} and \ref{T::3}, respectively, and
	{\tiny	\[-\lambda_3^p-\frac{h_m}{h^m}\left( \dfrac{1}{\sqrt{\delta_1}}\left( \delta_1\,\delta_2-\int_a^b(c(t)-c_m)\,dt\right)\right) \sqrt{\frac{\left( \frac{\pi}{b-a}\right) ^4+p\left( \frac{\pi}{b-a}\right) ^2}{b-a}}\leq c_m< -\left( \frac{\pi}{b-a}\right) ^4-\left( \frac{\pi}{b-a}\right)^2\,, \]}
	then  problem \eqref{Ec::Op},\eqref{Ec::cf} has a unique negative solution in $X$.
	\end{theorem}

\begin{proof}
	The existence of a unique solution, $u\in X$, is given by Theorem \ref{T::3}.
	
	To see that $u<0$ in $(a,b)$, we assume that $c_m<-\lambda_3^p$. On the contrary, if $c_m\geq -\lambda_3^p$, using Theorem \ref{T::3} we know that $T[p,c]$ is strongly inverse negative and the result follows directly.
	
	Let $e(t):=\max\left\lbrace c(t),-\lambda_3^p\right\rbrace $ be a continuous function on $ I$ such that $-\lambda_3^p\leq e(t)\leq c^m$, and we write the equation \eqref{Ec::Op} in an equivalent form
	\[L[p,e]\,u=h(t)-\left( c(t)-e(t)\right) \,u\,,\]
	which, from Theorem \ref{T::3}, is an strongly inverse negative operator on $X$, and we consider the recurrence formula
	\[L[p,e]\,u_{n+1}=h(t)-\left( c(t)-e(t)\right) \,u_{n}\,,\quad n=0,1,2,\dots\]

As in the proof of Theorem \ref{T::4}, we choose $u_0\equiv 0$ and we have $T[p,e]\,u_1=h(t)>0$, then $u_1<0$ on $(a,b)$, $u_1'(a)<0$ and $u_1'(b)>0$.
	
	Since, $u_1$ is the unique solution of  problem  $T[p,e]\,u=h(t)$ in $X$, using the Theorem \ref{T::3} we have
	\[\left\| u_1\right\|_{C( I)}\leq \dfrac{h^m }{\dfrac{1}{\sqrt{\delta_1}}\left( \delta_1\,\delta_2-\int_a^b(c(t)-c_m)\,dt\right) }\sqrt{\dfrac{b-a}{\left( \dfrac{\pi}{b-a}\right) ^4+p\left( \dfrac{\pi}{b-a}\right) ^2}}\,.\]
	
	Fixing 
	\[R:=\dfrac{1 }{\dfrac{1}{\sqrt{\delta_1}}\left( \delta_1\,\delta_2-\int_a^b(c(t)-c_m)\,dt\right)}\sqrt{\dfrac{b-a}{\left( \dfrac{\pi}{b-a}\right) ^4+p\left( \dfrac{\pi}{b-a}\right) ^2}}\,,\]
	taking into account that $0\geq c(t)-e(t)\geq c_m+\lambda_3^p$ and that $u_1<0$, we arrive to
	{\small \[T[p,e]u_2=h(t)-\left( c(t)-e(t)\right) u_{1}\geq h_m-(c_m+\lambda_3^p){u_1}_m=h_m-(c_m+\lambda_3^p)\left\| u_1\right\|_{C( I)}\gneqq0\,.\]}
	
	The rest of the proof follows the same steps as \cite[Theorem 5]{Dra1}
\end{proof}	

\begin{remark}
	Realize that in this case we cannot extrapolate in a direct way the results to the set $\tilde X$. This is due to the fact that we are not able to ensure the existence and constant sign of functions $y_{p,c}^a$ and $x_{p,c}^a$ in this new situation.
	
	However in Theorem \ref{T::4} (1), since we do not reach any eigenvalue, we can apply Lemma \ref{L::14}. Moreover, from Proposition \ref{P::Usm1} we can deduce the boundedness needed and the result remains valid for $T[p,c]$ in $\tilde X$.
\end{remark}

\end{document}